\documentclass[a4paper,10pt]{article}
\usepackage[utf8]{inputenc}
\usepackage[margin=1in]{geometry}
\usepackage{amssymb}
\usepackage{amsmath}
\usepackage{amsthm}
\usepackage{hyperref}
\usepackage{etoolbox}

\newtheorem{thrm}{Theorem}[section]
\newtheorem{cor}[thrm]{Corollary}
\newtheorem{conj}[thrm]{Conjecture}
\newtheorem{lma}[thrm]{Lemma}
\newtheorem{propn}[thrm]{Proposition}

\theoremstyle{definition}{
\newtheorem{defn}[thrm]{Definition}

}
\newcommand{\Imm}{\operatorname{Im}}
\newcommand{\gl}{\operatorname{GL}_2^+(\mathbb{Q})}

\newcommand{\uh}{\mathcal{H}}
\newcommand{\suq}{\subseteq}
\title{A Note on the Degree of Field Extensions Involving Classical and Nonholomorphic Singular Moduli}
\author{Haden Spence}

\begin{document}

\maketitle

\begin{abstract}
 In their 2015 paper \cite{Mertens2015}, Mertens and Rolen prove that for a certain level 6 ``almost holomorphic'' modular function $P$, the degree of $P(\tau)$ over $\mathbb{Q}$ for quadratic $\tau$ is as large as expected, settling a conjecture of Bruinier and Ono \cite{Bruinier2013}.  Analogously for level 1 modular functions $f$, we expect $\mathbb{Q}(f(\tau))$ to have similar degree to $\mathbb{Q}(j(\tau))$.  In this paper, I show for a wide class of level 1 almost holomorphic modular functions that
 \[\dfrac{1}{M}[\mathbb{Q}(j(\tau)):\mathbb{Q}]\leq [\mathbb{Q}(f(\tau)):\mathbb{Q}]\leq[\mathbb{Q}(j(\tau)):\mathbb{Q}]\]
 for all quadratic $\tau$ and some constant $M$.  This is proven using techniques of o-minimality, and hence can easily be made uniform; the constant $M$ depends only upon the ``degree'' of $f$ (in a certain well-defined sense).
\end{abstract}
\section{Introduction}
Let $P_\delta^N$ be the (finite) set of primitive, reduced, integer quadratic forms
\[Q(x,y)=Ax^2+Bxy+Cy^2,\]
with $N|A$, of discriminant $\delta$.  Let $\tau_Q$ be the unique root of $Q(\tau,1)$ lying in the upper half plane.  In their 2015 paper \cite{Mertens2015}, Mertens and Rolen prove that, for a certain almost holomorphic modular (AHM) function $P$ of level 6, the so-called \emph{class polynomial}
\[H^P_\delta=\prod_{Q\in P^6_\delta}\left(X-P(\tau_Q)\right)\in\mathbb{Q}[X]\]
is irreducible over $\mathbb{Q}$ whenever $\delta \equiv 1\mod 24$, thus settling a question posed by Bruinier and Ono in \cite{Bruinier2013}.  In particular, we have
\[[\mathbb{Q}(P(\tau)):\mathbb{Q}]=\#P^N_\delta\]
whenever $\tau$ is a root of a quadratic polynomial $6Ax^2+Bx+C$ with discriminant congruent to 1 mod 24.

%
%
%
\bigskip
In this note, we will be looking at similar results for some level 1 AHM functions.  A focus point is the function
\[\chi^*=1728\cdot\dfrac{E_2^*E_4E_6}{E_4^3-E_6^2},\]
where $E_4$ and $E_6$ are the usual Eisenstein series and $E_2^*$ is the so-called \emph{almost holomorphic} Eisenstein series
\[E_2^*(\tau)=E_2(\tau)-\dfrac{3}{\pi\Imm\tau}.\]
The function $\chi^*$ is a level 1 AHM function and has been studied or has arisen incidentally in various places such as \cite{Griffin2015}, \cite{Masser1975}, \cite{Zagier2008}.  Together with $j$, this function $\chi^*$ generates the field $F^*$ of level 1 AHM functions, that is 
\[F^*=\mathbb{C}(j,\chi^*).\]
For those unfamiliar with AHM functions, this will suffice for the purposes of this paper as a definition of $F^*$.  For a more intrinsic definition including functions of higher level, see the excellent survey by Zagier, \cite{Zagier2008}.

Various facts about the arithmetic properties of $\chi^*$ are known.  Elsewhere, I have investigated the ``special sets'' of $\chi^*$, in the context of an Andr\'e-Oort type result; see \cite{Spence2016}.  For this note we will be focussing only on the special \emph{points} of AHM functions.

It is known, thanks to Masser \cite{Masser1975}, that $\chi^*(\tau)$ is an algebraic number when $\tau$ is quadratic.  Indeed, Masser shows that $\mathbb{Q}(\chi^*(\tau))\suq\mathbb{Q}(j(\tau))$ for all quadratic points.  This is in contrast with the case considered by Mertens and Rolen; the function $P$ considered there, having level greater than 1, does \emph{not} satisfy such an inclusion of fields.  Nonetheless, many of the techniques of Mertens and Rolen carry over nicely to level 1 AHM functions like $\chi^*$.  Indeed, a paper of Braun, Buck and Girsch \cite{Braun2015} uses the same number-theoretic techniques to extend to the Mertens/Rolen result to a wide variety of AHM functions.  In light of this, we might expect a result like the following.
\begin{conj}\label{conj:Equality}
 The class polynomial of $\chi^*$, namely
\[H^{\chi^*}_\delta=\prod_{Q\in P^1_\delta}\left(X-\chi^*(\tau_Q)\right)\]
is irreducible over $\mathbb{Q}$, and hence $\mathbb{Q}(j(\tau))=\mathbb{Q}(\chi^*(\tau))$.  
\end{conj}
Existing results, particularly from the paper of Braun, Buck and Girsch \cite{Braun2015}, get very close to this.  The techniques involve a lot of powerful number-theoretic machinery, centred around two fairly simple-sounding ingredients.  

The function $\chi^*$ may be written as a sum of two $q$-expansions, yielding
\[\chi^*(\tau)=\left(1-\dfrac{3}{\pi\Imm \tau}\right)q^{-1}+E(q),\]
where $q=e^{2\pi i\tau}$ and $E$ is an error term consisting of nonnegative powers of $q$.  The first ingredient is to get some sensible absolute bound on the size of $E$ as $\tau$ varies in $\mathbb{F}$.  This gives the result for $\tau$ of sufficiently large discriminant.  The second ingredient is the calculation of $\chi^*(\tau)$ for the finitely many quadratic $\tau$ of small discriminant.  The smaller the bound found for $E$, the fewer such calculations are needed.  

I will approach this problem from a rather different direction.  Using the theory of o-minimal structures, a branch of model theory, one can make good steps towards statements like \ref{conj:Equality}.  The central result of this note is of this type; we prove (a stronger version of) the following.
\begin{thrm}\label{thrm:WeakBoundedness}
 Let $f\in\mathbb{Q}(j,\chi^*)$ be nonconstant.  Then there is a number $M\in\mathbb{N}$ such that, for any quadratic $\tau\in\uh$, we have
 \[[\mathbb{Q}(j(\tau)):\mathbb{Q}(f(\tau))]\leq M.\]
\end{thrm}
This theorem may not be entirely new.  The number-theoretic techniques of Mertens and Rolen, while not applied to $\chi^*$ specifically, implicitly involves finding such bounds; their work could likely be extended fairly easily to get something like \ref{thrm:WeakBoundedness}.  Our main theorem, however, does seem to be new.  Using the Andr\'e-Oort theorem from \cite{Spence2016} and uniformity ideas of Scanlon \cite{Scanlon2004}, we are able to get a uniform version of \ref{thrm:WeakBoundedness}; this uniformity does not seem to be attainable with the methods of Mertens of Rolen. 
 \begin{defn}
  For a field $F$ and natural numbers $d$ and $n$, define
  \[F^{\leq d}(X_1,\dots,X_n)\]
  to be the set of rational functions in the $X_i$, with coefficients in $F$, which may be written as $f/g$, for polynomials $f$ and $g$ of degree at most $d$.
 \end{defn}
 With this definition, we can state the uniform result which is our main theorem for this note.
\begin{thrm}\label{thrm:BoundednessOfExtension}
 For each natural number $d$, there is a constant $M_d$ such that, whenever $f\in \mathbb{Q}^{\leq d}(j,\chi^*)$ is nonconstant and $\tau\in\uh$ is quadratic, we have
 \[[\mathbb{Q}(j(\tau)):\mathbb{Q}(f(\tau))]\leq M_d.\]
\end{thrm}
A remark on this: it can be viewed as a form of independence result, or orthogonality result.  It says, essentially, that the special points of $\chi^*$ and the special points of $j$, though they lie in the same fields and have similar degrees over $\mathbb{Q}$, cannot be ``globally dependent'' on one another.  A further remark we must also make is that the bound $M_d$ relies on the Siegel bound for class numbers, and hence is certainly not effective.  

As discussed, we will be proving this using methods of o-minimality, in particular the Pila-Wilkie Theorem; the unfamiliar reader can see \cite{Dries1998}, \cite{Pil} and \cite{Pila2011} for details about these techniques.  In the next section, we compile some technical lemmas and propositions which will be necessary for the proof.  These are not particularly o-minimal in nature; the Pila-Wilkie arguments appear in Section \ref{sect:conc}.

\bigskip\noindent
\textbf{Acknowledgements.} Although this is not a long paper, it owes much to several people.  I would like to thank my supervisor Jonathan Pila, who unfailingly provides invaluable guidance and support.  I would also like to thank Harry Schmidt for several productive conversations on the topics of this paper and for setting me down this route of investigation in the first place.  Finally I would like to thank my father Derek for our many discussions and his keen proof-reading eyes!


\section{Technicalities}
Out first technical results are quite number-theoretic in nature, of a similar feel to results from the work of Mertens and Rolen \cite[Proposition 3.2]{Mertens2015}.  These were also proven independently by myself in \cite{Spence2016}, using the work of Masser \cite{Masser1975}.
\begin{propn}\label{propn:GaloisControl}
 Let $\tau\in\uh$ be a quadratic point.  Let $\sigma$ be a Galois automorphism acting on (the splitting field of) $\mathbb{Q}(j(\tau))$ over $\mathbb{Q}$.  Let $\tau'\in\uh$ be quadratic such that $\sigma(j(\tau))=j(\tau')$.  Then $\sigma(\chi^*(\tau))=\chi^*(\tau')$. 
 \begin{proof}
  See \cite[Proposition 5.2]{Spence2016}.
 \end{proof}
\end{propn}

\begin{cor}\label{cor:ClassPolysPowerOfIrreducible}
 For any nonconstant $f\in\mathbb{Q}(j,\chi^*)$, the class polynomial of $f$,
 \[H^f_\delta(x)=\prod_{Q\in P_\delta}\left(x-f(\tau_Q)\right),\]
 (which is defined over $\mathbb{Q}$ for such $f$) decomposes over $\mathbb{Q}$ as a power of an irreducible polynomial.  That is,
 \[H^f_\delta=p^{k_\delta}\]
 for some natural number $k_\delta$ and some polynomial $p\in\mathbb{Q}[X]$, irreducible over $\mathbb{Q}$.  
 \begin{proof}
  Follows from the previous proposition by elementary Galois theory.
 \end{proof}
\end{cor}

The next few results we will need are of a rather different feel; somewhat more geometric in nature. They are motivated by ideas from model theory, specifically the now-standard Pila-Zannier strategy to proving diophantine results using o-minimality.  The familiar reader may recognise the first result as an Ax-Lindemann-type result for $j$ and $\chi^*$.  It was proven in \cite{Spence2016} as a main ingredient in the proof of an Andr\'e-Oort result for $j$ and $\chi^*$.
\begin{thrm} 
 Let $\pi$ be the map from $\uh^n$ to $\mathbb{C}^{2n}$ defined by
 \[(\tau_1,\dots,\tau_n)\mapsto(j(\tau_1),\chi^*(\tau_1),\dots,j(\tau_n),\chi^*(\tau_n)).\]
 Let $S$ be an arc of a real algebraic curve in $\uh^n$ and suppose that $S\suq \pi^{-1}(V)$, where $V$ is some irreducible variety in $\mathbb{C}^{2n}$.  Then $S$ is contained in a weakly $\uh$-special variety $G$ with $G\suq \pi^{-1}(V)$.
\end{thrm}
Here, a ``weakly $\uh$-special'' variety is any subset of $\uh^n$ cut out by equations either of the form
\begin{itemize}
 \item $\tau_i=g\tau_j$, with $g\in\gl$, or
 \item $\tau_i=c$, with $c\in\uh$.
\end{itemize}
All that is needed for our purposes is the following fact: the only proper weakly $\uh$-special subvarieties of $\uh$ are points.
\begin{cor}\label{cor:algArcsForAllFs}
 Let $f\in\mathbb{C}(j,\chi^*)$ be nonconstant.  There does not exist a real algebraic arc $S\suq\uh$ upon which $f$ is constant.
 \begin{proof}
  If $f$ were constant on a real algebraic arc $S$, we could then find a polynomial $p$ in two variables, nontrivial, such that $p(j(\tau),\chi^*(\tau))=0$.  By the previous theorem, we would get a weakly $\uh$-special variety $G$ containing $S$ such that $p(j(\tau),\chi^*(\tau))=0$ for all $\tau\in G$.  The variety $G$ is necessarily positive-dimensional, whence $G=\uh$.  Hence $p(j(\tau),\chi^*(\tau))=0$ for all $\tau\in\uh$, so $p$ must be trivial since $j$ is holomorphic but $\chi^*$ is not.  Contradiction.
 \end{proof}
\end{cor}
This most recent result will later give us some control over preimages of the form $f^{-1}(z)$, $z\in\mathbb{C}$.  On its own, however, it will not be enough; to give us further control over such preimages, we will need some information about the \emph{Jacobian} of $f$, which by definition is the map $J_f:\uh\to\mathbb{R}$,
\[J_f(x+iy)=\det\begin{pmatrix}\dfrac{\partial}{\partial x}\operatorname{Re}f(x+iy)&\dfrac{\partial}{\partial y}\operatorname{Re}f(x+iy)\\\dfrac{\partial}{\partial x}\Imm f(x+iy)&\dfrac{\partial}{\partial y}\Imm f(x+iy)
 \end{pmatrix}.\]
 It is well-known that if $f$ is constant on some positive-dimensional path, then $J_f$ will vanish on that path.  The result we need, therefore, is something to give us control over the zero set of $J_f$.
\begin{lma}\label{lma:Jacobians}
 Let $h\in\mathbb{R}(j,\chi^*)$ be nonconstant.  Then the Jacobian $J_h$ of $h$ is not identically zero.
 \begin{proof}
  Such an $h$ may be written as
  \[h(\tau)=\dfrac{f}{g}=\dfrac{\sum_k f_k(\Imm\tau)^{-k}}{\sum_k g_k(\Imm\tau)^{-k}},\]
  for some holomorphic functions $f_k, g_k$ having $q$-expansions with real coefficients.  A tedious manipulation then shows, for $\tau=iy$, that $J_h$ vanishes only if
  \begin{equation}\label{eqn:JacobianFirst}f\sum_k g_k'(\Imm\tau)^{-k}=g\sum_k f_k'(\Imm\tau)^{-k}\end{equation}
  or
  \begin{equation}\label{eqn:JacobianSecond}f\left(\sum_k g_k'(\Imm\tau)^{-k}-\sum_k kg_k(\Imm\tau)^{-k-1}\right)=g\left(\sum_k f_k'(\Imm\tau)^{-k}-\sum_k kf_k(\Imm\tau)^{-k-1}\right).\end{equation}
  By growth considerations, the coefficient of each power of $\Imm\tau$ must vanish individually.  In case (\ref{eqn:JacobianFirst}) we compare the $k=0$ terms to get
  \[f_0g_0'-g_0f_0'=0,\]
  whence $(f/g)'=0$ on $\tau=iy$ (and therefore everywhere), so $f_0=\lambda g_0$ for some constant $\lambda$.  By an isomorphism theorem from \cite{Zagier2008}, this implies $f=\lambda g$, which we are assuming is not the case.  
  
  The case (\ref{eqn:JacobianSecond}) is exactly the same by comparing the $k=0$ terms; the sums of the form $\sum kg_k(\Imm\tau)^{-k-1}$ contribute nothing to the $k=0$ term.
 \end{proof}
\end{lma}
With these technicalities out of the way, we can progress to the proof of \ref{thrm:WeakBoundedness}.
\section{Proving Theorem \ref{thrm:WeakBoundedness}}\label{sect:conc}
We will make significant use of techniques from model theory.  One crucial fact is the following: any $f\in\mathbb{C}(j,\chi^*)$ is \emph{definable} in $\mathbb{R}_{\text{an,exp}}$ when restricted to the standard fundamental domain
\[\mathbb{F}=\left\{\tau\in\uh:-\frac{1}{2}\leq\operatorname{Re}(\tau)\leq\frac{1}{2}\text{ and }|z|>1\right\}.\]
In particular, if $P$ is a definable property, sets of the form
\[Z=\{\tau\in\mathbb{F}:P(f(\tau))\text{ holds}\},\]
will always be definable.  Readers unfamiliar with definability in o-minimal structures can see \cite{Dries1998}, \cite{Pil} and \cite{Pila2011} for details. 

\bigskip
The bulk of the work towards proving \ref{thrm:WeakBoundedness} lies in the following lemma.
\begin{lma}\label{lma:CountingQuadraticPreimages}
   For each nonconstant $f\in\mathbb{Q}(j,\chi^*)$, there is a constant $M$ such that
   \[\#\{\tau\in\mathbb{F}:f(\tau)=z,\tau\text{ is quadratic.}\}\leq M\]
   for all $z\in\mathbb{C}$.
 \begin{proof}
  Consider the definable set
  \[Z=\Big\{\tau\in\mathbb{F}: \dim_{\mathbb{R}}\big(\mathbb{F}\cap f^{-1}\{f(\tau)\}\big)=1\Big\}.\]
  If $f$ is constant on a set of positive real dimension, then $J_f$ also vanishes on that set.  Hence
  \[Z\suq \{\tau\in\mathbb{F}:J_f(\tau)=0\}.\]
  By \ref{lma:Jacobians}, $J_f$ is not identically zero, so by analytic continuation the zero set of $J_f$ can have no interior.  Hence the definable set $Z$ must consist of just finitely many points and real analytic arcs.  (In particular, note that $f(Z)$ is finite; this will be useful later.)
  
  On $\mathbb{F}\setminus Z$, the function $f$ must be finite-to-one, by definability and the Cell Decomposition Theorem.  By uniform boundedness, there is $m\in\mathbb{N}$ such that, for all $\tau
  \in\uh$,
  \[\#\big(\mathbb{F}\cap f^{-1}\{f(\tau)\}\big)\leq m\text{ unless }\tau\in Z.\] 
%
  So it will be sufficient to prove the following.
  
  \bigskip  
  \emph{Claim.}  $Z$ contains only finitely many quadratic points.
  
  \emph{Proof of Claim.}
  By Corollary \ref{cor:algArcsForAllFs}, the set $Z$ can never contain an arc of a real algebraic curve.  So, by the Pila-Wilkie Theorem \cite{Pila2006}, we have (for all $\epsilon>0$) a constant $c_{\epsilon}$ such that
  \begin{equation}\label{eqn:PilaWilkieBound}\#\{\tau\in Z:\tau\text{ is quadratic, }\operatorname{Ht}(\tau)\leq T\}\leq c_{\epsilon}T^{\epsilon},\end{equation}
  where $\operatorname{Ht}$ denotes the \emph{absolute multiplicative height} of an algebraic point, as defined in \cite{Bombieri2006}.  As is fairly standard in this area, we will be playing this bound off against lower bounds provided by Galois considerations.  The lower bound is derived from the well-known inequality
  \[[\mathbb{Q}(j(\tau)):\mathbb{Q}]\geq c_\nu D^{\frac{1}{2}-\nu},\]
  where $\tau$ is a quadratic point of discriminant $\delta$, $D=|\delta|$ and $c_\nu$ is some constant depending only on $\nu$.  This, as usual, is derived from the Siegel lower bound for class numbers of quadratic orders; see \cite{Siegel1935} for the original result and \cite{Pila2011} for typical applications.  Let us fix $\nu=1/4$, so that $c=c_{1/4}$ is an absolute constant.
  
  We begin work on this lower bound as follows.  First, let $\tau$ be a quadratic point of discriminant $\delta$, with $f(\tau)=z$.  Recall from \ref{cor:ClassPolysPowerOfIrreducible} that $H^f_\delta(x)=p_z(x)^{k_\delta}$, where $p_z$ is the minimal polynomial of $z$.  We have
  \[[\mathbb{Q}(j(\tau)):\mathbb{Q}]=k_\delta[\mathbb{Q}(f(\tau)):\mathbb{Q}],\]
  whence $k_\delta\geq \frac{cD^{\frac{1}{4}}}{[\mathbb{Q}(z):\mathbb{Q}]}$.
  
  To each power of $p_z(x)$ arising in $H^f_\delta(x)$, there corresponds a new point $\tau'$, of discriminant $\delta$, with $f(\tau')=z$, since $z$ is a root of $p_z(x)$.  Hence there are at least $k_\delta$ distinct quadratic points $\tau$, of discriminant $\delta$, such that $f(\tau)=z$.   If we take, say, $\epsilon=\frac{1}{8}$ in the Pila-Wilkie bound (\ref{eqn:PilaWilkieBound}), we will get a contradiction whenever $D=|\delta|$ is greater than some constant $\kappa$.  This $\kappa$ will depend only on $[\mathbb{Q}(f(\tau)):\mathbb{Q}]$. 
   
  Since $f(Z)$ is a finite set, and $f(\tau)\in f(Z)$, so in particular $[\mathbb{Q}(f(\tau)):\mathbb{Q}]$ is bounded above by a constant depending only on $f$.  So the constant $\kappa$ above depends only on $f$.\end{proof}  
\end{lma}
Now the proof of \ref{thrm:WeakBoundedness} is easy.

 \begin{proof}[Proof of Theorem \ref{thrm:WeakBoundedness}]
  Let $f\in\mathbb{Q}(j,\chi^*)$ be nonconstant and consider the class polynomial of $f$:
  \[H^f_\delta(x)=\prod_{Q\in P_\delta}\left(x-f(\tau_Q)\right),\]  which by Corollary \ref{cor:ClassPolysPowerOfIrreducible} is a power of an irreducible polynomial.  Say $H^f_\delta(x)=p(x)^{k_\delta}$.  Then $p$ is necessarily the minimal polynomial of $f(\tau)$ for some (any) quadratic $\tau$ of discriminant $\delta$.

  Since $H^f_\delta=p^{k_\delta}$, there must be a root of $H^f_\delta$ of order $k_\delta$.  This means that there are (at least) $k_\delta$ distinct quadratic points $\tau_1,\dots,\tau_{k_\delta}\in\mathbb{F}$ with $f(\tau_i)=f(\tau_1)$ for all $i$.  By \ref{lma:CountingQuadraticPreimages}, there can be at most $M$ such quadratic points, whence $k_\delta\leq M$.
  
  For a quadratic point $\tau$ of discriminant $\delta$, the degree of $j(\tau)$ over $\mathbb{Q}$ is equal to the degree of $H^f_\delta$, while the degree of $f(\tau)$ over $\mathbb{Q}$ is equal to the degree of $p$.  Hence we must have
  \[[\mathbb{Q}(j(\tau)):\mathbb{Q}(f(\tau))]=\dfrac{\deg H^f_\delta}{\deg p} = k_\delta\leq M,\]
  as required.
 \end{proof}
 
\section{Uniformity}
  In this section we will prove the uniform version of \ref{thrm:WeakBoundedness}, namely Theorem \ref{thrm:BoundednessOfExtension}.  For this we require a certain amount of setup.  One crucial component is the Andr\'e-Oort result from \cite{Spence2016}, which we will state shortly.  First we need a definition.
  \begin{defn}
    A point in $\mathbb{C}^{2n}$ is called $(j,\chi^*)$-special if it takes the form $(j(\tau),\chi^*(\tau))$, for some quadratic $\tau$.
  \end{defn}
  The result also discusses $(j,\chi^*)$-special subvarieties of $\mathbb{C}^{2n}$.  The precise definition of these is not important; they are particular subvarieties of $\mathbb{C}^{2n}$ containing Zariski dense sets of $(j,\chi^*)$-special points.  We will only need to know the following.
  \begin{itemize}
   \item The $(j,\chi^*)$-special subvarieties of $\mathbb{C}^{2}$ are precisely the $(j,\chi^*)$-special points and $\mathbb{C}^2$ itself.
   \item The fibre $V_a$ of a $(j,\chi^*)$-special subvariety $V\suq\mathbb{C}^{2(m+n)}$ at a $(j,\chi^*)$-special point $a\in\mathbb{C}^{2n}$ is a union of finitely many $(j,\chi^*)$-special subvarieties of $\mathbb{C}^{2m}$; the number of connected components of $V_a$ depends only on $V$, not on $a$.
   \item If the fibre $V_a$ of a $(j,\chi^*)$-special subvariety $V\suq\mathbb{C}^{2+2n}$ at any $(j,\chi^*)$-special point $a\in\mathbb{C}^{2n}$ is equal to $\mathbb{C}^2$, then: for every $(j,\chi^*)$-special $b\in\mathbb{C}^{2n}$, $V_b$ is either empty or $\mathbb{C}^2$.
  \end{itemize}
  The Andr\'e-Oort theorem we need is the following, which was proven in \cite{Spence2016}.
  \begin{thrm}\label{thrm:AOforPi}
    Let $V\suq\mathbb{C}^{2n}$ be an algebraic variety.  Then $V$ contains only finitely many maximal $(j,\chi^*)$-special subvarieties.
  \end{thrm}
  We will be combining this with the following theorem of Scanlon, which is Lemma 3.2 from his paper \cite{Scanlon2004}.
  \begin{thrm}\label{lma:UniformConstructibleOverQ}
    Let $K$ be an algebraically closed field, $X$ and $B$ algebraic varieties over $K$, $Y\suq X\times B$ a constructible subset, and $A\suq X(K)$.  There is a natural number $n$ and a constructible set $Z\suq X\times X^n$ such that, for any parameter $b\in B(K)$, there is some $a\in A^n$ for which $Y_b(K)\cap A=Z_a(K)\cap A$.
  \end{thrm}
  We will be applying this with $A$ being the set of $(j,\chi^*)$-special points in $X=\mathbb{C}^{2}$.  It essentially does all of the uniformisation work for us; the bulk of the proof of \ref{thrm:BoundednessOfExtension} lies in the following uniform version of Lemma \ref{lma:CountingQuadraticPreimages}.
  
  \begin{lma}\label{lma:UniformCountingQuadraticPreimages}
    For each natural number $d$, there is a constant $M_d$ such that, whenever $f\in\mathbb{Q}^{\leq d}(j,\chi^*)$ is nonconstant and $z\in\mathbb{C}$, we have
    \[\#\{\tau\in\mathbb{F}:f(\tau)=z,\tau\text{ is quadratic}.\}\leq M_d.\]
    \begin{proof}
      For each $d$, there is a natural number $n(d)$ such that $\mathbb{C}^{n(d)}$ parametrises elements of $\mathbb{C}^{\leq d}(X,Y)$ in the obvious way.  For $b\in\mathbb{C}^{n(d)}$, we will write $f_b(X,Y)$ for the corresponding element of $\mathbb{C}^{\leq d}(X,Y)$.
      
      Define a variety $V\suq\mathbb{C}^{2+n(d)+1}$ by
      \[(X,Y,b,z)\in V\Longleftrightarrow f_b(X,Y)=z.\]
      Lemma \ref{lma:UniformConstructibleOverQ}, applied to $Y$ (with $X=\mathbb{C}^2$, $B=\mathbb{C}^{n(d)+1}$ and $A$ equal to the set of $(j,\chi^*)$-special points), yields a natural number $N$ and a constructible set $Z\suq \mathbb{C}^{2+2N}$, such that for any parameter $(b,z)\in\overline{\mathbb{Q}}^{n(d)+1}$, there is a $(j,\chi^*)$-special point $a\in\mathbb{C}^{2N}$ such that $V_{b,z}$ and $Z_a$ (both subvarieties of $\mathbb{C}^2$) have the same $(j,\chi^*)$-special points.
      
      We can write
      \[Z=Z_1\setminus(Z_2\setminus(\dots\setminus Z_k))\]
      for some varieties $Z_i$.  
      
      Apply Theorem \ref{thrm:AOforPi} to each $Z_i$.  Then each $Z_i$ contains only finitely many maximal $(j,\chi^*)$-special subvarieties; call the collection of such subvarieties $\sigma(Z_i)$.
      
      For some $i$, it might be the case that $S\in\sigma(Z_i)$ contains $\mathbb{C}^{2}\times a$, for some $(j,\chi^*)$-special point $a\in\mathbb{C}^{2N}$.  However, by Theorem \ref{thrm:WeakBoundedness}, there are only finitely many $(j,\chi^*)$-special points in any $Z_a$ with $a$ corresponding to $b\in\mathbb{Q}^{n(d)}$ and $z\in\overline{\mathbb{Q}}$.  Hence we may safely ignore such $S$ and assume that no $S\in\sigma(Z_i)$ contains any $\mathbb{C}^2\times a$.
      
      The number of $(j,\chi^*)$-special points in $Z_a$ is then bounded above by
      \[M_d=\max\{\#S_a:S\in \sigma(Z_i),i\leq k.\}.\]
      By the properties of $(j,\chi^*)$-special subvarieties, (since we have excluded any $S$ for which $S_a=\mathbb{C}^2$ for any $a$) each $S_a$ is finite and bounded in size independently of $a$, so this is a well-defined maximum independent of $a$.          
		
      So take any $b\in\mathbb{Q}^{n(d)}$ and any $z\in\overline{\mathbb{Q}}$.  Let $f=f_b(j,\chi^*)\in\mathbb{Q}^{\leq d}(j,\chi^*)$.  Let $a\in\mathbb{C}^{2N}$ be the $(j,\chi^*)$-special point afforded by Lemma \ref{lma:UniformConstructibleOverQ}.  The quadratic points $\tau\in\mathbb{F}$ such that $f(\tau)=z$ are in one-to-one correspondence with the $(j,\chi^*)$-special points of $V_{b,z}$.  So the number of such $\tau$ is equal to the number of $(j,\chi^*)$-special points in $Z_a$, which is bounded above by $M_d$.              
    \end{proof}
  \end{lma}
  Now the proof of \ref{thrm:BoundednessOfExtension} goes exactly as the proof of Theorem \ref{thrm:WeakBoundedness} from Lemma \ref{lma:CountingQuadraticPreimages}.
  \begin{proof}[Proof of Theorem \ref{thrm:BoundednessOfExtension}]
    Let $f\in\mathbb{Q}^{\leq d}(j,\chi^*)$.  Recall that, for each discriminant $\delta$, we have $H_\delta^f=p_\delta^{k_\delta}$, where $p_\delta$ is irreducible, and is therefore necessarily the minimal polynomial of $f(\tau)$ for some (any) quadratic $\tau$ with $\delta(\tau)=\delta$.  It follows that $\{\tau_Q:Q\in P_\delta^1\}$ contains $k_\delta$ distinct quadratic points $\tau_1,\dots,\tau_{k_\delta}\in\mathbb{F}$ with $f(\tau_i)=f(\tau_1)$ for all $i$.  By Lemma \ref{lma:UniformCountingQuadraticPreimages}, there are at most $M_d$ such points, whence $k_\delta\leq M_d$.
    
    Finally,
    \[[\mathbb{Q}(j(\tau)):\mathbb{Q}(f(\tau))]=\dfrac{\deg H_{\delta(\tau)}^f}{\deg p_{\delta(\tau)}}=k_\delta\leq M_d.\]
  \end{proof}

\section{Further Work}
So far we have made no mention whatsoever of modular functions of level other than 1.  The methods demonstrated here should be quite applicable to modular functions with level, up to the inclusion of some currently missing ingredients.  The simplest case for higher level AHM functions does not require much.  Let $\Gamma_N^\mathbb{Q}$ be the field of meromorphic modular functions of level $N$, with rational q-expansions (including at infinity).  Thanks to work of Mertens/Rolen \cite{Mertens2015} and Schertz \cite{Schertz2002}, the equivalent of Proposition \ref{propn:GaloisControl} holds for functions in the field $\Gamma_N^\mathbb{Q}(\chi^*)$, giving us the necessary Galois information.  Definability follows from the q-expansions, so the only remaining detail here is to get some control on the Jacobians of the relevant functions.  A suitable analogue of Lemma \ref{lma:Jacobians} should not be too difficult to attain, at which point the methods would yield 
\[[\mathbb{Q}(f(\tau)):\mathbb{Q}]\geq \dfrac{\# P_\delta^N}{M}\]
where $f=p(\chi^*,g_1,\dots,g_k)$, where the $g_i$ are the generators of $\Gamma_N^\mathbb{Q}$ and $p$ is some rational function.  Moreover, the constant $M$ would depend only upon $N$ and the degree of $p$.  The function $P$ considered by Mertens and Rolen lies in $\Gamma_N^\mathbb{Q}(\chi^*)$, so the above would yield a weak generalisation of their result.  However, the story of AHM functions with level does not end there.  

Consider, for instance, a level $N$ meromorphic modular form $f$, of weight $2$.  Then the function
$\dfrac{E_2^*}{f}$
is certainly an AHM function of level $N$, but need not lie in $\Gamma_N^\mathbb{Q}(\chi^*)$.  In this situation, we would as usual need some analogue of \ref{lma:Jacobians}, giving us some control over the zero sets of the Jacobians of the relevant functions.  This, once again, should not be too difficult.  The difficulty lies instead in the Galois information.  It is not at all clear that any analogue of \ref{propn:GaloisControl} exists, since the work of Mertens, Rolen and Schertz does not apply.  To progress with this problem, one would likely need to extend the work of Schertz into the context of AHM functions.  This would seem to be by far the most significant challenge presenting itself in attempting to extend the results of this paper to functions of higher level.  I will postpone further discussion of this problem for later work.

\bibliographystyle{../../bib/scabbrv}
\bibliography{../../bib/thebib}
\end{document}